\documentclass[12pt,reqno]{amsart}

\usepackage{amsmath, amsthm, amsopn, amssymb, enumerate}
\usepackage{verbatim} 

\newtheorem{thm}{Theorem}[section]

\newtheorem{lemma}[thm]{Lemma}
\newtheorem{cor}[thm]{Corollary}

\newtheorem{prop}[thm]{Proposition}
\newtheorem{obs}[thm]{Observation}
\newtheorem{fact}{Fact}
\newtheorem{claim}{Claim}

\newtheorem*{claim*}{Claim}

\theoremstyle{definition}

\newtheorem*{AlgBasic}{The Basic Algorithm}
\newtheorem*{MainAlg}{The Main Algorithm}
\newtheorem{defn}[thm]{Definition}

\newcommand{\ds}{\displaystyle}

\newcommand{\G}{\mathcal{G}}

\newcommand{\C}{\mathbb{C}}
\newcommand{\Cs}{\mathbb{C}^*}

\newcommand{\Gh}{\hat{G}}
\newcommand{\SF}{\mathrm{SF}}
\newcommand{\chit}{\chi_{T}}
\newcommand{\rng}{\mathrm{range}}

\renewcommand{\Re}{\mathrm{Re}}

\def\B{\mathcal{B}}

\def\F{\mathcal{F}}
\def\HH{\mathcal{H}}

\def\P{\mathcal{P}}
\def\Q{\mathcal{Q}}

\def\N{\mathbb{N}}
\def\Pr{\mathbb{P}}

\def\ZZ{\mathbb{Z}}

\def\le{\leqslant}
\def\ge{\geqslant}

\def\eps{\varepsilon}
\def\<{\langle}
\def\>{\rangle}

\def\Bin{\textup{Bin}}

\title{Counting sum-free sets in Abelian groups}

\author{Noga Alon}
\address{School of Mathematical Sciences, Tel Aviv University, Tel Aviv 69978, Israel} \email{nogaa@post.tau.ac.il}

\author{J\'ozsef Balogh}
\address{Department of Mathematics, University of Illinois, 1409 W. Green Street, Urbana, IL 61801} \email{jobal@math.uiuc.edu}

\author{Robert Morris} 
\address{IMPA, Estrada Dona Castorina 110, Jardim Bot\^anico, Rio de Janeiro, RJ, Brasil} \email{rob@impa.br}

\author{Wojciech Samotij}
\address{School of Mathematical Sciences, Tel Aviv University, Tel Aviv 69978, Israel; and Trinity College, Cambridge CB2 1TQ, UK} \email{ws299@cam.ac.uk}

\thanks{Research supported in part by: (NA) ERC Advanced Grant DMMCA, 
a USA-Israeli BSF grant and the Israeli I-Core program; 
(JB) NSF CAREER Grant DMS-0745185, UIUC Campus Research Board Grant 11067, and OTKA Grant K76099; (RM) a CNPq bolsa de Produtividade em Pesquisa; (WS) ERC Advanced Grant DMMCA and a Trinity College JRF}

\begin{document}

\begin{abstract}
In this paper we study sum-free sets of order $m$ in finite Abelian groups. We prove a general theorem on 3-uniform hypergraphs, which allows us to deduce structural results in the sparse setting from stability results in the dense setting. As a consequence, we determine the typical structure and asymptotic number of sum-free sets of order $m$ in Abelian groups $G$ whose order is divisible by a prime $q$ with $q \equiv 2 \pmod 3$, for every $m \ge C(q) \sqrt{n \log n}$, thus extending and refining a theorem of Green and Ruzsa. In particular, we prove that almost all sum-free subsets of size $m$ are contained in a maximum-size sum-free subset of $G$. We also give a completely self-contained proof of this statement for Abelian groups of even order, which uses spectral methods and a new bound on the number of independent sets of size $m$ in an $(n,d,\lambda)$-graph.
\end{abstract}

\maketitle

\section{Introduction}

An important trend in Combinatorics in recent years has been the formulation and proof of various `sparse analogues' of classical extremal results in Graph Theory and Additive Combinatorics. Due to the recent breakthroughs of Conlon and Gowers~\cite{CG} and Schacht~\cite{Sch}, many such theorems, e.g., the theorems of Tur{\'a}n~\cite{Turan} and Erd{\H{o}}s and Stone~\cite{ErSt} in extremal graph theory, and the theorem of Szemer{\'e}di~\cite{SzAP} on arithmetic progressions, are now known to extend to sparse random sets. For structural and enumerative results, such as the theorem of Kolaitis, Pr{\"o}mel and Rotshchild~\cite{KoPrRo} which states that almost all $K_{r+1}$-free graphs are $r$-colorable, perhaps the most natural sparse analogue is a corresponding statement about subsets of a given fixed size $m$, whenever $m$ is not too small. In this paper, we prove such a result in the context of sum-free subsets of Abelian groups and provide a general framework for solving problems of this type. To be precise, we obtain a sparse analogue of a result of Green and Ruzsa~\cite{GR05}, which describes the structure of a typical sum-free subset of an Abelian group.

Sparse versions of classical extremal and Ramsey-type results were first proved for graphs by Babai, Simonovits and Spencer~\cite{BSS}, and for additive structures by Kohayakawa, {\L}uczak and R\"odl~\cite{KLR}, and in recent years there has been a tremendous interest in such problems (see, e.g,~\cite{FRRT,RR1,RR2}). Mostly, these results have been in the random setting; for example, Graham, R\"odl and Ruci\'nski~\cite{GRR} showed that if $p \gg 1/\sqrt{n}$, and $B \subseteq \ZZ_n$ is a $p$-random subset\footnote{A $p$-random subset of a set $X$ is a random subset of $X$, where each element is included with probability $p$, independently of all other elements.}, then with high probability every $2$-colouring of $B$ contains a monochromatic solution of $x + y = z$. The extremal version of this question was open for fifteen years, until it was recently resolved by Conlon and Gowers~\cite{CG} and Schacht~\cite{Sch}. 

For problems of the type we are considering, results are known only in a few special cases. Most notably, Osthus, Pr\"omel and Taraz~\cite{OPT}, confirming (and strengthening) a conjecture of Pr\"omel and Steger~\cite{PS96}, proved that if $m \ge \big( \frac{\sqrt{3}}{4} + \eps \big) n^{3/2} \sqrt{\log n}$ then almost all triangle-free graphs with $m$ edges are bipartite; moreover, the constant $\sqrt{3}/4$ is best possible. This result can be seen as a sparse version of the classical theorem of Erd\H{o}s, Kleitman and Rothschild~\cite{EKR}, which states that almost all triangle-free graphs are bipartite. A similarly sharp result was proved by Friedgut, R\"odl, Ruci\'nski and Tetali~\cite{FRRT} for the existence of monochromatic triangles in two-colourings of $G_{n,p}$. It is an interesting open problem to prove such a sharp threshold in the setting of Theorem~\ref{thm:groups}, below.

A set $A \subseteq G$, where $G$ is an Abelian group, is said to be \emph{sum-free} if $(A + A) \cap A = \emptyset$, or equivalently, if there is no solution to the equation $x + y = z$ with $x,y,z \in A$. Sum-free subsets of Abelian groups are central objects of interest in Additive Combinatorics, and have been studied intensively in recent years. The main questions are as follows: What are the largest sum-free subsets of $G$? How many sum-free sets are there? And what does a typical such set look like? Over forty years ago, Diananda and Yap~\cite{DY} determined the maximum density $\mu(G)$ of a sum-free set in $G$ whenever $|G|$ has a prime factor $q \not\equiv 1 \pmod 3$, but it was not until 2005 that Green and Ruzsa~\cite{GR05} completely solved this extremal question for all finite Abelian groups. On the second and third questions, Lev, \L uczak and Schoen~\cite{LLS} and Sapozhenko~\cite{Sap02} determined the asymptotic number of sum-free subsets in an Abelian group of even order by showing that almost all such sets\footnote{We say that almost all sets in a family $\F$ of subsets of $G$ satisfy some property $\P$ if the ratio of the number of sets in $\F$ that have $\P$ to the number of all sets in $\F$ tends to $1$ as $|G|$ tends to infinity.} lie in the complement of a subgroup of index~$2$. Green and Ruzsa~\cite{GR05} extended this result to Abelian groups which have a prime factor $q \equiv 2 \pmod 3$, and showed also that a general finite Abelian group $G$ has $2^{(1 + o(1))\mu(G)|G|}$ sum-free subsets.

We say that $G$ is of Type~I if $|G|$ has a prime divisor $q$ with $q \equiv 2 \pmod 3$, and Type~I($q$) if $q$ is the smallest such prime. Diananda and Yap~\cite{DY} proved that if $G$ is of Type~I($q$), then $\mu(G) = (q+1)/(3q)$; moreover, they described all sum-free subsets of $G$ with $\mu(G)|G|$ elements. Green and Ruzsa~\cite{GR05} determined the asymptotic number of sum-free subsets of an Abelian group $G$ of Type I, by showing that almost every sum-free set in $G$ is contained in some sum-free set of maximum size. Balogh, Morris and Samotij~\cite{BMS} studied $p$-random subsets of such groups, and showed that if $p \ge C(q) \sqrt{\log n / n}$, $G$ is an Abelian group of Type I($q$) and order $n$, and $G_p$ is a $p$-random subset of $G$, then with high probability every maximum-size sum-free subset of $G_p$  is contained in some sum-free subset of $G$ of maximum size. In the case $G = \ZZ_{2n}$, they showed that if $p \ge \big( \frac{1}{\sqrt{3}} + \eps \big) \sqrt{n \log n}$ and $A \subseteq G$ is a $p$-random subset, then with high probability the unique largest sum-free subset of is $A \cap O_{2n}$, where $O_{2n} \subseteq \ZZ_{2n}$ denotes the set of odd residues modulo $2n$. Moreover, the constant $1/\sqrt{3}$ in this result is best possible. 

Let us denote by $\SF(G,m)$ the collection of sum-free subsets of size $m$ in a finite Abelian group $G$. The following theorem refines the result of Green and Ruzsa~\cite{GR05} to sum-free sets of fixed size $m$, provided that $m \ge C(q) \sqrt{n \log n}$. 

\begin{thm}\label{thm:groups}
For every prime $q \equiv 2 \pmod 3$, there exists a constant $C(q) > 0$ such that the following holds. Let $G$ be an Abelian group of Type $I(q)$ and order $n$, and let $m \ge C(q) \sqrt{n \log n}$. Then almost every sum-free subset of $G$ of size $m$ is contained in a maximum-size sum-free subset of $G$, and hence
$$|\SF(G,m)| \, = \, \lambda_q  \cdot \left(\# \big\{ \text{elements of $G$ of order $q$} \big\} + o(1)\right) {\mu(G)n \choose m}$$ 
as $n \to \infty$, where $\lambda_q = 1$ if $q = 2$ and $\lambda_q = 1/2$ otherwise.
\end{thm}

Although the factor $\lambda_q \cdot \# \{\text{elements of } G \text{of order } q \}$ above may appear mysterious, it is a natural consequence of the characterization of maximum-size sum-free sets in groups of
Type I, see Theorem~\ref{thm:SFG-structure}.  We remark that the lower
bound $m \ge C(q)\sqrt{n \log n}$ is sharp up to a constant factor,
since there are at least $(n/2) \big( \mu(G) n - 3m\big)^{m-1} / (m-1)!$
sum-free subsets of $G$ which contain exactly one element outside a given
maximum-size sum-free subset of $G$, and this is larger than $m^{1/5}
{\mu(G) n \choose m}$ if $m \le \frac{1}{5} \sqrt{n \log n}$. (Here,
and throughout, $\log$ denotes the natural logarithm.) Hence, assuming
$m^{1/5}$ is much larger than the number of elements of order $q$ in $G$, almost no sum-free subset of $G$ of this size is contained in a maximum-size sum-free subset of $G$.

We shall prove Theorem~\ref{thm:groups} using a new theorem (see Section~\ref{GenThmSec}) which describes the structure of a typical independent set in a 3-uniform hypergraph $\HH$ that satisfies a certain natural `stability' property, see Definition~\ref{def:aB}. The key ingredient in the proof of this theorem is a new method of enumerating independent sets in 3-uniform hypergraphs. We shall also use a simplified version of this method to prove a new bound on the number of independent sets in a certain class of expander graphs known as $(n,d,\lambda)$-graphs.

First, let us recall the definition of $(n,d,\lambda)$-graphs, which are
an important class of expanders; for a detailed introduction to expander
graphs, we refer the reader to~\cite{AS} or~\cite{HLW}. Given a graph
$\G$, let $\lambda_1 \ge \ldots \ge \lambda_n$ denote the eigenvalues of
the adjacency matrix of $\G$. We call $\max\{ |\lambda_2|,|\lambda_n|\}$
the \emph{second eigenvalue} of $\G$.

\begin{defn}[$(n,d,\lambda)$-graphs]
A graph $\G$ is an $(n,d,\lambda)$-graph if it is $d$-regular, 
has $n$ vertices, and the absolute value 
of each of its nontrivial eigenvalues  is
at most $\lambda$. 
\end{defn}

Alon and R{\"o}dl~\cite{AR} gave an upper bound on the number of independent sets in an $(n,d,\lambda)$-graph, and used their result to give sharp bounds on multicolour Ramsey numbers. When $\lambda = \Omega(d)$ (as $n \to \infty$), Theorem~\ref{thm:graphs} below provides a significantly stronger bound than that of Alon and R{\"o}dl, for a wider range of $m$; it is moreover asymptotically sharp. In fact, we will not assume anything about the second eigenvalue of a graph $\G$ as our bound on the number of independent sets of $\G$ will depend only on the \emph{smallest} eigenvalue of $\G$. Given a graph $\G$, let $\lambda(\G)$ be the smallest eigenvalue of the adjacency matrix of $\G$ (denoted by $\lambda_n$ above) and let $I(\G,m)$ be the number of independent sets of size $m$ in $\G$. Observe that $\lambda(\G) < 0$ for every non-empty $\G$ and that, by definition, every $(n,d,\lambda)$-graph satisfies $\lambda(\G) \ge -\lambda$.

\begin{thm}\label{thm:graphs}
  For every $\eps > 0$, there exists a constant $C = C(\eps)$ such that the following holds. If $\G$ is an $n$-vertex $d$-regular graph with $\lambda(\G) \ge -\lambda$, then
  \[
  I(\G,m) \le {\left( \frac{\lambda}{d+\lambda} + \eps \right) n \choose m}
  \]
  for every $m \ge Cn/d$.
\end{thm}

We remark that the constant $\frac{\lambda}{d+\lambda}$ in Theorem~\ref{thm:graphs} is best possible, since there exist $n$-vertex $d$-regular graphs with $\lambda(\G) \ge -\lambda$ and $\alpha(\G) = \frac{\lambda}{d+\lambda} n$ for many values of $n$, $d$ and $\lambda$ (here $\alpha(\G)$ denotes the independence number of $\G$). For example, consider a blow-up of the complete graph $K_{t+1}$, where each vertex is replaced by a set of size $n/(t+1)$ and each edge is replaced by a random $d/t$-regular bipartite graph with colour classes of size $n/(t+1)$ each. This blown-up graph $\G$ is $d$-regular and, with high probability, it satisfies $\lambda(\G) = -d/t$ and $\alpha(\G) = n/(t+1)$.


In Section~\ref{q=2Sec}, we shall use Theorem~\ref{thm:graphs}, together with some basic facts about characters of finite Abelian groups, to give a completely self-contained proof of Theorem~\ref{thm:groups} in the case $q = 2$. For previous results relating the problem of estimating the number of sum-free subsets of groups, to that of estimating the number of independent sets in regular graphs, see for example~\cite{Noga,LLS,Sap02}. For other results on counting independent sets in graphs and hypergraphs, see Balogh and Samotij~\cite{BSmm,BSst}, Carroll, Galvin and Tetali~\cite{CGT}, Galvin and Kahn~\cite{GK}, Kahn~\cite{Kahn,K02}, Peled and Samotij~\cite{PS}, Sapozhenko~\cite{Sap01} and Zhao~\cite{Zhao}.

The rest of the paper is organised as follows. In Section~\ref{GenThmSec}, we state our structural theorem for 3-uniform hypergraphs, and in Section~\ref{GraphSec} we prove Theorem~\ref{thm:graphs}. In Sections~\ref{AlgSec} and~\ref{JansonSec}, we prove the structural theorem, and in Sections~\ref{GroupSec} we shall apply it to prove Theorem~\ref{thm:groups}. Finally, in Section~\ref{q=2Sec}, we shall prove Theorem~\ref{thm:groups} again in the case $q = 2$.

\section{A structural theorem for 3-uniform hypergraphs}\label{GenThmSec}

In this section we shall introduce our main tool: a theorem which allows us to deduce structural results for sparse sum-free sets from stability results in the dense setting. Since we wish to apply our result to sum-free sets in various Abelian groups, we shall use the language of general $3$-uniform (sequences of) hypergraphs $\HH = (\HH_n)_{n \in \N}$, where $|V(\HH_n)| = n$. Throughout the paper, the reader should think of $\HH$ as encoding the Schur triples (that is, triples $(x,y,z)$ with $x + y = z$) in an additive structure.

We now define the stability property with which we shall be able to work. Let $\alpha \in (0,1)$ and let $\B = (\B_n)_{n \in \N}$, where $\B_n$ is a family of subsets of $V(\HH_n)$. We shall write $|\B_n|$ for the number of sets in $\B_n$, and set $\|\B_n\| = \max\{ |B| : B \in \B_n\}$. 

\begin{defn}\label{def:aB}
A sequence of hypergraphs $\HH = (\HH_n)_{n \in \N}$ is said to be $(\alpha,\B)$-stable if for every $\gamma > 0$ there exists $\beta > 0$ such that the following holds. If $A \subseteq V(\HH_n)$ with $|A| \ge (\alpha - \beta) n$, then either $e(\HH_n[A]) \ge \beta e(\HH_n)$, or $|A \setminus B| \le \gamma n$ for some $B \in \B_n$.
\end{defn}

Roughly speaking, a sequence of hypergraphs $(\HH_n)$ is $(\alpha, \B)$-stable if for every $A \subseteq V(\HH_n)$ such that $|A|$ is almost as large as the extremal number for $\HH_n$ (i.e., the size of the largest independent set), the set $A$ is either very close to an extremal set $B \in \B_n$, or it contains many (i.e., a positive fraction of all) edges of $\HH_n$. Observe that classical stability results, such as that of Erd\H{o}s and Simonovits~\cite{ES1,ES2}, are typically of this form.

We shall need two further technical conditions on $\HH$. Let
$$\Delta_2(\HH_n) \,=\, \max_{T \subseteq V(\HH_n), \, |T| = 2} \big| \big\{ e \in \HH_n \colon T \subseteq e \big\} \big|,$$
and note that if $\HH$ encodes Schur triples then $\Delta_2(\HH_n) \le 3$. Also define 
$$\delta(\HH_n,\B_n) \, := \, \min_{B \in \B_n} \min_{v \in V(\HH_n) \setminus B} \left| \big\{ e \in \HH_n \,:\, |e \cap B| = 2 \text{ and } v \in e \big\} \right|,$$
and, as usual, write $\alpha(\HH_n)$ for the size of the largest independent set in $\HH_n$.

The following theorem is the key step in the proof of Theorem~\ref{thm:groups}.

\begin{thm}\label{genthm}
Let $\HH = (\HH_n)_{n \in \N}$ be a sequence of $3$-uniform hypergraphs which is $(\alpha,\B)$-stable, with $e(\HH_n) = \Theta(n^2)$ and $\Delta_2(\HH_n) = O(1)$. Suppose that $|\B_n| = n^{O(1)}$, that $\alpha(\HH_n) \ge \|\B_n\| \ge \alpha n$, and that $\delta(\HH_n,\B_n) = \Omega(n)$. Then there exists a constant $C = C(\HH,\B) > 0$ such that if
$$m \,\ge\, C\sqrt{n \log n},$$
then almost every independent set in $\HH_n$ of size $m$ is a subset of some $B \in \B_n$.
\end{thm}

We shall prove Theorem~\ref{genthm} in Sections~\ref{AlgSec} and~\ref{JansonSec}. In Section~\ref{GroupSec}, we shall use it to prove Theorem~\ref{thm:groups}.

\section{Independent sets in regular graphs with no small eigenvalues}\label{GraphSec}

As a warm-up for the proof of Theorem~\ref{genthm}, we shall prove a bound on the number of independent sets in regular graphs with no small eigenvalues, Theorem~\ref{thm:graphs}, which improves a theorem of Alon and R\"odl~\cite{AR}. This result will be a key tool in our self-contained proof of Theorem~\ref{thm:groups} in the case $q = 2$, see Section~\ref{q=2Sec}. Moreover, many of the ideas from the proof of Theorem~\ref{thm:graphs} will be used again in the proof of Theorem~\ref{genthm}. We remark that the technique of enumerating independent sets in graphs used in this section was pioneered by Kleitman and Winston~\cite{KW} and our proof of Theorem~\ref{thm:graphs}, below, requires little more than their original method.

Given a graph $\G$ on $n$ vertices, and an integer $m \in [n]$, let $I(\G)$ denote the number of independent sets in $\G$, and recall that $I(\G,m)$ denotes the number of independent sets of size $m$ in $\G$. Alon~\cite{Noga} proved that if $\G$ is a $d$-regular graph on $n$ vertices, then $I(\G) \le 2^{n/2 + o(n)}$ (as $d \to \infty$), resolving a conjecture of Granville (see~\cite{Noga}), and suggested that the unique $\G$ that maximizes $I(\G)$ among all such (i.e., $n$-vertex $d$-regular) graphs might be a disjoint union of copies of $K_{d,d}$. This conjecture was proven (using the entropy method) by Kahn~\cite{Kahn} for bipartite graphs, and recently in full generality by Zhao~\cite{Zhao}.

\begin{thm}[Kahn~\cite{Kahn}, Zhao~\cite{Zhao}] 
Let $\G$ be a $d$-regular graph on $n$ vertices. Then
$$I(\G) \, \le \, \big( 2^{d+1} - 1)^{n/2d},$$
where equality holds if and only if $\G$ is a disjoint union of copies of $K_{d,d}$.
\end{thm}

Since $-d$ is an eigenvalue of $K_{d,d}$, any $d$-regular graph $\G$ containing a copy of $K_{d,d}$ satisfies $\lambda(\G) = -d$. One might hope that a stronger bound on $I(\G)$ holds for $d$-regular graphs $\G$ with $\lambda(\G) > -d$. Alon and R\"odl~\cite{AR} proved such a bound on $I(\G,m)$ for the slightly narrower class of $(n,d,\lambda)$-graphs and used their result to give sharp bounds on Ramsey numbers.

\begin{thm}[Alon and R\"odl~\cite{AR}] \label{AlonRodl}
Let $\G$ be an $(n,d,\lambda)$-graph. Then
$$I(\G,m) \, \le \, \left( \frac{emd^2}{4\lambda n \log n} \right)^{2(n/d)\log n} {2\lambda n / d \choose m}$$
for every $m \ge 2 (n/d) \log n$. 
\end{thm}

If $\lambda = \Omega(d)$ as $n \to \infty$, then Theorem~\ref{thm:graphs} improves the above result in three ways: it provides a stronger bound for a wider range of values of $m$ in a wider class of graphs. Theorem~\ref{thm:graphs} is an immediate consequence of the following theorem, combined with the Alon-Chung lemma (Lemma~\ref{lemma:AC}, below).

\begin{thm}\label{thm:graphs-loc-dense}
For every $\eps,\delta > 0$, there exists a constant $C = C(\eps,\delta)$ such that the following holds. Let $\G$ be a $d$-regular graph on $n$ vertices, and suppose that $2e(A) \ge \eps |A|d$ for every $A \subseteq V(\G)$ with $|A| \ge \big( \alpha + \delta \big) n$. Then 
$$I(\G,m) \le {(\alpha + 2\delta)n \choose m}$$
for every $m \ge Cn/d$.
\end{thm}

The assumption in Theorem~\ref{thm:graphs-loc-dense} that $2e(A) \ge \eps |A| d$ might seem somewhat strong; however, it follows from the Expander Mixing Lemma that it is satisfied by every $(n,d,\lambda)$-graph with $\lambda/(d+\lambda) \le \alpha$. For two sets $S, T \subseteq V(\G)$, let $e(S,T)$ denote the number of pairs $(x,y) \in S \times T$ such that $\{x,y\} \in E(\G)$. In particular, we have $e(S,S) = 2e(S)$ for every $S \subseteq V(\G)$. The following result is proved\footnote{Although the result is stated in~\cite{AC} in a slightly different form, its proof there in fact implies Lemma~\ref{lemma:AC}.} in~\cite{AC}.

\begin{lemma}[Alon-Chung~\cite{AC}]\label{lemma:AC}
Let $\G$ be an $n$-vertex $d$-regular graph. Then for all $A \subseteq V(G)$,
\[
2e(A) \ge \frac{d}{n}|A|^2 + \frac{\lambda(\G)}{n} |A|\big(n - |A|\big).
\]
\end{lemma}

We first deduce Theorem~\ref{thm:graphs} from Theorem~\ref{thm:graphs-loc-dense} and Lemma~\ref{lemma:AC}.

\begin{proof}[Proof of Theorem~\ref{thm:graphs}]
We claim that if $\G$ is an $n$-vertex $d$-regular graph with $\lambda(\G) \ge -\lambda$ and $\alpha = \lambda/(d+\lambda)$, then $2e(A) \ge \eps|A|d$ for every $A \subseteq V(\G)$ with $|A| \ge (\alpha + \eps)n$. This implies that $\G$ satisfies the assumption of Theorem~\ref{thm:graphs-loc-dense} (with $\delta = \eps$), and so the theorem follows.

To prove the claim, suppose that $A \subseteq V(\G)$ satisfies $|A| \ge (\alpha + \eps)n$. By Lemma~\ref{lemma:AC},
$$2e(A) \; \ge \;  \frac{d}{n} \cdot |A|^2 \,-\, \frac{\lambda}{n} \cdot |A|\big(n - |A| \big) \; = \; \frac{|A|}{n} \Big[ \big( d + \lambda \big) |A| - \lambda n \Big].$$
Our lower bound on $|A|$ and the choice of $\alpha = \lambda/(d+\lambda)$ now give
$$2e(A) \; \ge \;  |A| \Big[ \big( d + \lambda \big) \big( \alpha + \eps \big) - \lambda \Big] \; = \; \eps  |A| \big( d + \lambda \big) \; \ge \; \eps |A| d,$$
as required.
\end{proof}

In the proof of Theorem~\ref{thm:graphs-loc-dense}, we shall use an algorithm which uniquely encodes every independent set $I$ of size $m$ in~$\G$ as a pair $(S,I \setminus S)$, where $|S| \le 2n / \eps d \le 2m / C\eps$, and $I \setminus S$ is contained in some set $A$, with $|A| \le (\alpha + \delta)n$, which depends only on $S$. At all times, it maintains a partition of $V(\G)$ into sets $S$, $X$, and $A$ (short for Selected, eXcluded, and Available), such that $S \subseteq I \subseteq A \cup S$. 

At each stage of the algorithm, we will need to order the vertices of $A$ with respect to their degrees. For the sake of brevity and clarity of the presentation, let us make the following definition.

\begin{defn}[Max-degree order]
Given a graph $\G$ and a set $A \subseteq V(\G)$, the \emph{max-degree order} on $A$ is the following linear order $(v_1,\ldots,v_{|A|})$ on the elements of $A$: For every $i \in \{1, \ldots, |A|\}$, $v_i$ is the maximum-degree vertex in the graph $\G[A \setminus \{v_1, \ldots, v_{i-1}\}]$; we break ties by giving preference to vertices that come earlier in some predefined ordering of $V(\G)$.
\end{defn}

We are now ready to describe the Basic Algorithm.

\begin{AlgBasic}
Set $A = V(\G)$ and $S = X = \emptyset$. Now, while $|A| > (\alpha + \delta)n$, we repeat the following:
\begin{enumerate}
\item[$(a)$] Let $i$ be the minimal index (in the max-degree order on $A$) such that $v_i \in I$. 
\item[$(b)$] Move $v_i$ from $A$ to $S$.
\item[$(c)$] Move $v_1, \ldots, v_{i-1}$ from $A$ to $X$ (since they are not in $I$ by the choice of $i$).
\item[$(d)$] Move $N(v_i)$ from $A$ to $X$ (since $I$ is independent and $v_i \in I$). 
\end{enumerate}
Finally, when $|A| \le (\alpha + \delta)n$, we output $S$ (which is a subset of $V(\G)$) and $I \setminus S$ (which is a subset of $A$). 
\end{AlgBasic}

We remark that, as well as in~\cite{KW}, algorithms similar to the one above have been considered before to bound the number of independent sets in graphs~\cite{AR,Sap01} and hypergraphs~\cite{BSmm,BSst}. 

\begin{proof}[Proof of Theorem~\ref{thm:graphs-loc-dense}]
The theorem is an easy consequence of the following two statements:
\begin{equation}\label{eq:t}
I(\G,m) \,\le\, \ds\sum_{t = 1}^{t_0} {n \choose t}{(\alpha + \delta)n \choose m-t},
\end{equation}
where $t_0 = \ds\frac{n}{\eps d} + 1$, and
\begin{equation}\label{eq:bin}
{n \choose t}{(\alpha + \delta)n \choose m-t} \; \le \; \frac{1}{m} {(\alpha + 2\delta)n \choose m}
\end{equation}
if $t \le 2n / \eps d$ and $m \ge Cn/d$. We shall prove~\eqref{eq:t} using the Basic Algorithm;~\eqref{eq:bin} follows from a straightforward calculation.

To prove~\eqref{eq:t}, let $t \in \N$ be the number of elements of $S$ at the end of the algorithm, and note that we have at most ${n \choose t}$ choices for $S$. Now, crucially, we observe that $A$ is uniquely determined by $S$ (given the original ordering of $V(\G)$); indeed, step $(a)$ of the Basic Algorithm requires no knowledge of $I$, only of $S$ and $\G$. Since $|A| \le (\alpha + \delta)n$, it follows that we have at most ${(\alpha + \delta)n \choose m-t}$ choices for $I \setminus S$.

It will thus suffice to show that, given our assumptions on $\G$, the algorithm terminates in at most $t_0 = \frac{n}{\eps d} + 1 \le \frac{2n}{\eps d}$ steps. We shall show that $A$ loses at least $\eps d$ elements at each step of the algorithm (except perhaps the last), from which this bound follows immediately. Indeed, consider a step of the algorithm (not the last), in which a vertex $v_i$ is moved to $S$, and set $A' = A \setminus \{v_1, \ldots, v_{i-1}\}$. Since this is not the last step, we have $|A'| \ge (\alpha + \delta)n$, and so $2e(A') \ge \eps|A'|d$, by our assumption on $\G$. Thus $|N(v_i) \cap A'| \ge \eps d$, since $v_i$ is the vertex of maximum degree in $\G[A']$, and hence $A$ loses at least $\eps d$ elements in this step, as claimed.  

To prove~\eqref{eq:bin}, we use the fact that, if $t$ is replaced by $t+1$, then the left-hand side is multiplied by
\begin{equation}
  \label{eq:ratio}
 \frac{n-t}{t+1} \cdot \frac{m-t}{(\alpha + \delta)n - m + t + 1}.
 \end{equation}
We claim that~\eqref{eq:ratio} is at most $\frac{2}{\delta} \cdot \left(\frac{m}{t}\right)^2$. To prove this, we consider two cases: if $m \le (\alpha + \delta/2)n$, then~\eqref{eq:ratio} is at most 
$$\frac{n}{t} \cdot \frac{2m}{\delta n} \,= \,\frac{2m}{\delta t}$$
while if $m > (\alpha + \delta/2)n$ then it is at most
$$\frac{n}{t} \cdot \frac{m}{t}  \, \le \, \frac{2}{\delta} \cdot \left(\frac{m}{t}\right)^2,$$
since $n \le 2m/ \delta$, and our assumptions imply that $\alpha(\G) < (\alpha + \delta)n$, so we may assume that $m < (\alpha + \delta)n$. Thus, for each $t$ with $t \le t_0 \le \frac{2n}{\eps d}$,
$$\binom{n}{t} \binom{(\alpha + \delta)n}{m-t} \, \le \, \binom{(\alpha+\delta)n}{m} \prod_{r=1}^t \, \frac{2}{\delta} \cdot \left(\frac{m}{r}\right)^2 \, \le \, \left( \frac{1}{t_0!} \right)^2 \left( \frac{2m}{\delta} \right)^{2t_0} \binom{(\alpha + \delta)n}{m}.$$
Since ${b \choose c} \le \big( \frac{b}{a} \big)^c {a \choose c}$ for every $a > b > c \ge 0$, $k! > (k/e)^k$ for every $k \in \N$, the function $t_0 \mapsto \big(\frac{em}{\delta t_0}\big)^{t_0}$ is increasing on the interval $(0,m)$,  and $t_0 \le \frac{2n}{\eps d} \le \frac{2m}{C\eps}$, this is at most
$$\left( \frac{2em}{\delta t_0} \right)^{2t_0} \binom{(\alpha + \delta)n}{m}  \, \le \, \left( \frac{C\eps e}{\delta} \right)^{\frac{4}{C\eps}m} \left(\frac{\alpha + \delta}{\alpha + 2\delta}\right)^m \binom{(\alpha + 2\delta)n}{m}  \, \le \, \frac{1}{m} \cdot \binom{(\alpha + 2\delta)n}{m}$$
as required. In the final inequality we used the fact that $C$ is sufficiently large as a function of $\delta$ and $\eps$, and that $m$ is sufficiently large (as a function of $\delta$), since $m \ge Cn/d$.
\end{proof}

\section{Algorithm argument}\label{AlgSec}

In this section, we shall introduce a more powerful algorithm than that used in Section~\ref{GraphSec}. We shall use this algorithm in the proof of Theorem~\ref{genthm} to bound the number of independent sets which contain at least $\delta m$ elements of $V(\HH_n) \setminus B$ for every $B \in \B_n$. We shall show that when $m \gg \sqrt{n}$, then the number of such independent sets is exponentially small. The model example that the reader should keep in mind when reading this section is when $\HH_n$ is the hypergraph of Schur triples in an $n$-element Abelian group of Type~I($q$), where $q$ is some prime satisfying $q \equiv 2 \pmod 3$.

Given a hypergraph $\HH_n$, a family of sets $\B_n$, and $\delta > 0$, we define
$$\SF^{(\delta)}_{\ge}(\HH_n,\B_n,m) \, := \, \Big\{ I \in \SF(\HH_n,m) \,\colon\, |I \setminus B| \ge \delta m \textup{ for every } B \in \B_n \Big\},$$
where $\SF(\HH_n,m)$ denotes the collection of independent sets in $\HH_n$ of size $m$. The following theorem shows that there are few independent sets in $\HH_n$ (i.e., sum-free sets) of size $m$ which are far from every set $B \in \B_n$.

\begin{thm}\label{algprop}
Let $\alpha > 0$ and let $\HH = (\HH_n)_{n \in \N}$ be a sequence of $3$-uniform hypergraphs which is $(\alpha,\B)$-stable, has $e(\HH_n) = \Theta(n^2)$ and $\Delta_2(\HH_n) = O(1)$. If $\|\B_n\| \ge \alpha n$, then for every $\delta > 0$, there exists a $C > 0$ such that the following holds. If $m \ge C\sqrt{n}$ and $n$ is sufficiently large, then 
$$\big| \SF^{(\delta)}_{\ge}(\HH_n,\B_n,m) \big| \,\le \, \Big( 2^{-\eps m} + \delta^m |\B_n| \Big) {\|\B_n\| \choose m}$$
for some $\eps = \eps(\HH, \delta) > 0$.
\end{thm}

We shall describe an algorithm which encodes every independent set $I$ in $\HH_n$ and produces short output for every $I \in \SF^{(\delta)}_{\ge}(\HH_n, \B_n, m)$. As before, our algorithm will maintain a partition of $V(\HH_n)$ into sets $S$, $X$, and $A$ (short for Selected, eXcluded, and Available), such that $S \subseteq I \subseteq A \cup S$. We shall also maintain a set $T \subseteq S$ (for Temporary), and the corresponding graph $\G_T$, i.e., the graph with vertex set $V(\HH_n)$, and edge set
$$E(\G_T) \,=\, \Big\{ \{u,v\} \subseteq V(\G_T) \,\colon\, \{u,v,w\} \in \HH_n \textup{ for some } w \in T \Big\}.$$ 
We shall frequently consider the max-degree order (defined in Section~\ref{GraphSec}) on the vertices of the graph $\G_T[A]$. 

\subsection{The Algorithm}

The idea of the algorithm is quite simple: we apply the Basic Algorithm of Section~\ref{GraphSec} to the graph $\G_T[A]$ as long as it is reasonably dense. If $\G_T[A]$ becomes too sparse, then there are four possibilities: either we have arrived at a set $A$ which has at most $(\alpha - \beta)n$ elements, or a set $A$ which is almost contained in some $B \in \B_n$; or if not, then we can use the $(\alpha,\B)$-stability of $\HH$ to find either a new set $T$ for which $\G_T$ is dense (see Case~2 below), or a set of linear size that contains very few elements of $I$. Therefore, after moving relatively few vertices of $\HH_n$ to $S$, our choice for $I \setminus S$ is limited to a set $A \subseteq V(\G)$ that is either small or almost contained in some $B \in \B_n$. It follows that if $I$ was far from every $B \in \B_n$, then (in both cases) the number of ways to choose $I \setminus S$ from $A$ is very small.

We begin by choosing some constants. Let $\gamma > 0$ and note that since $\HH$ is $(\alpha,\B)$-stable, there exists $\beta > 0$ so that if $|A| \ge (\alpha - \beta) |V(\HH_n)|$ and $|A \setminus B| > \gamma |V(\HH_n)|$ for every $B \in \B_n$, then $e(\HH_n[A]) \ge \beta e(\HH_n)$. Let us choose $\beta > 0$ sufficiently small so that $e(\HH_n) \ge \beta n^2$ and $\Delta_2(\HH_n) \le 1/\beta$ for all sufficiently large $n$. Let $C = C(\beta) > 0$ be sufficiently large, and set
$$d \,=\, \frac{Cn}{m} \,\le\, \frac{m}{C} \,\le\, \frac{n}{C}.$$
We are ready to describe the Main Algorithm; this is the key step in our proof of Theorem~\ref{genthm}.

\begin{MainAlg}
We initiate the algorithm with $T = S \subseteq I$, a deterministically chosen subset of $I$ of size $d$ (the first $d$ elements of $I$ in our ordering of $V(\HH_n)$, say), and with $A = V(\HH_n) \setminus S$ and $X = \emptyset$. Now, while $|A| > (\alpha - \beta)n$ and $|A \setminus B| > \gamma n$ for every $B \in \B_n$, we repeat the following steps:

\medskip
\noindent \textbf{Case 1:} If the average degree in $\G_T[A]$ is at least $\beta^4 d$, then:
\begin{enumerate}
\item[$(a)$] Let $i$ be the minimal index in the max-degree order on $V(\G_T[A])$ such that $v_i \in I$. 
\item[$(b)$] Move $v_i$ from $A$ to $S$.
\item[$(c)$] Move $v_1, \ldots, v_{i-1}$ from $A$ to $X$ (since they are not in $I$ by the choice of $i$).
\item[$(d)$] Move $N(v_i)$ from $A$ to $X$ (since $I$ is independent and $v_i \in I$). 
\end{enumerate}

\medskip
\noindent \textbf{Case 2:} If the average degree of $\G_T[A]$ is less than $\beta^4 d$, then we find a new set $T$ as follows. Since $\HH$ is $(\alpha,\B)$-stable, $|A| > (\alpha - \beta)n$, and $|A \setminus B| > \gamma n$ for every $B \in \B_n$, then $A$ contains at least $\beta e(\HH_n)$ edges of $\HH_n$. Set
$$Z \,=\, \Big\{ z \in A \,\colon\, e\big( \G_z[A] \big) \ge \beta^2 n \Big\},$$
where $\G_z[A]$ is the graph with vertex set $A$ and edge set $\big\{ \{x,y\} : \{x,y,z\} \in E(\HH_n) \big\}$. We call the elements of $Z$ \emph{useful}. Since $e(\HH_n) \ge \beta n^2$, we have $e\big( \HH_n[A] \big) \ge \beta^2 n^2$, and so
$$\sum_{z \in A} e\big( \G_z[A] \big) = 3e(\HH_n[A]) \ge 3\beta^2 n^2.$$
Moreover, we have $e\big( \G_z[A] \big) \le \Delta(\HH_n) \le \Delta_2(\HH_n) n \le n / \beta$ for every $z \in V(\G)$. Thus, by the pigeonhole principle, it follows that $|Z| \ge 2\beta^3 n$.

Now we have two subcases:
\begin{enumerate}
\item[$(a)$] If $I$ contains fewer than $d$ useful elements, then move these elements from $A$ to $S$ and move the other useful elements from $A$ to $X$.
\item[$(b)$] If $I$ contains more than $d$ useful elements, choose $d$ of them $u_1,\ldots,u_d$ (the first $d$ in our ordering, say) and move them from $A$ to $S$.  Moreover, set $T = \{u_1,\ldots,u_d\}$.
\end{enumerate}

\medskip
\noindent
Finally, when $|A| \le (\alpha - \beta)n$, or $|A \setminus B| \le \gamma n$ for some $B \in \B_n$, then we output $S$ (which is a subset of $V(\G)$) and $I \setminus S$ (which is a subset of $A$) and stop. 
\end{MainAlg}

\smallskip
We shall show that the Main Algorithm encodes at most $\beta^2m$ elements of $I$ in $S$, and that $S$ determines $A$. Theorem~\ref{algprop} then follows from some simple counting.

\subsection{Proof of Theorem~\ref{algprop}}

We begin by proving three straightforward claims about the Main Algorithm; these, together with some simple counting, will be enough to prove the theorem. The following statements all hold under the assumptions of Theorem~\ref{algprop}.

\begin{claim}\label{MA1}
The Main Algorithm passes through Case $1$ at most $2n/(\beta^4 d)$ times, and through Case $2$ at most $1/\beta^5$ times.
\end{claim}

\begin{proof}
We prove the second statement first. To do so, simply observe that each time we pass through Case $2(a)$, we move at least $2\beta^3 n - d \ge \beta^3 n$ vertices from $A$ to $X$, and each time we pass through Case $2(b)$, we obtain a graph $\G_T[A]$ with at least $\beta^2 n d / \Delta_2(\HH_n) - O(d^2) \ge 2\beta^4 n d$ edges. In the latter case, we must remove at least $\beta^4 n d$ edges from $\G_T[A]$ before we can return to Case $2$. Since $\Delta_2(\HH_n) \le 1/\beta$, it follows that $\Delta(\G_T) \le |T|/\beta = d/\beta$, since if $xy \in E(\G_T)$ then there exists $z \in T$ such that $\{x,y,z\} \in E(\HH_n)$, and for each pair $\{x,z\}$ there are at most $\Delta_2(\HH_n)$ such $y$. Thus we must remove at least $\beta^5 n$ vertices from $A$ before returning to Case~$2$, and hence the algorithm can pass through Case~$2$ at most $1/\beta^5$ times before the set $A$ shrinks to size $\alpha n$, as claimed.

To prove the first statement, note that each time we pass through Case~1 on two successive steps of the algorithm, we remove at least $\beta^4 d$ vertices of $A$ in the first of these. Indeed, since $\G_T[A]$ (for the second step) has average degree at least $\beta^4 d$, then by the definition of the max-degree order, the vertex we removed in the first step must have had forward degree at least $\beta^4 d$. By the argument above, there are at most $1/\beta^5$ steps at which this fails to hold, and therefore the algorithm passes through Case~$1$ at most
$$\frac{n}{\beta^4 d} \,+\, \frac{1}{\beta^5} \, \le \, \frac{2n}{\beta^4 d}$$
times, as claimed.
\end{proof}

The next claim is a simple consequence of Claim~\ref{MA1} and our choice of $d$.

\begin{claim}\label{MA2}
If $C \ge 3/\beta^7$, then $|S| \le \beta^2 m$ at the end of the Main Algorithm.
\end{claim}

\begin{proof}
Each time the Main Algorithm passes through Case~1, $|S|$ increases by one; each time it passes through Case~2, $|S|$ increases by at most $d$. Thus, by Claim~\ref{MA1} and our choice of $d$, 
$$|S| \, \le \, d + \frac{2n}{\beta^4 d} + \frac{d}{\beta^5}  \, \le \, \frac{m}{C} + \frac{2m}{\beta^4 C} + \frac{2m}{C \beta^5} \, \le \, \beta^2 m$$
if $C \ge 3/\beta^7$, as claimed.
\end{proof}

We next make the key observation that the set $S$ contains all the information we need to recover the final set $A$ produced by the algorithm.

\begin{claim}\label{MA3}
The set $A$ is uniquely determined by the set $S$ of selected elements. 
\end{claim}

\begin{proof}
This follows because all steps of the Main Algorithm are deterministic, and every element of $I$ which we need to observe is placed in $S$. Indeed, in Case~1 we observe only that $v_i \in I$, and that the elements $v_1,\ldots,v_{i-1} \not\in I$. Since $v_i \in S$ and $v_1\ldots,v_{i-1} \not\in S$, this can be deduced from $S$. In Case~$2$, the set $Z$ does not depend on $I$. If at most $d-1$ elements of $Z$ are in $S$, then we are in Case~$2(a)$ and the remaining elements of $Z$ are in $X$; otherwise, we are in Case~$2(b)$ and the first $d$ elements of $S \cap Z$ (in the order on $V(\HH_n)$) form the set $T$. Thus, inductively, we see that at each stage of the algorithm, the set $A$ is determined by the set $S$. 
\end{proof}

After all this preparation, we are ready to prove Theorem~\ref{algprop}.

\begin{proof}[Proof of Theorem~\ref{algprop}]
Let $\HH = (\HH_n)_{n \in \N}$ be an $(\alpha,\B)$-stable sequence of 3-uniform hypergraphs as in the statement of the theorem, let $\delta > 0$ be arbitrary, and choose $\gamma = \gamma(\alpha, \delta) > 0$ to be sufficiently small. Since $\HH$ is $(\alpha,\B)$-stable, there exists $\beta > 0$ such that if $|I| \ge (\alpha - 2\beta) n$, then either $|I \setminus B| < \gamma n$ for some $B \in \B_n$, or $I$ is not an independent set in $\HH_n$. In particular, note that if $\gamma < \delta /(\alpha - 2\beta)$, then either $m \le \left( \alpha - 2\beta \right) n$ or $\SF^{(\delta)}_{\ge}(\HH_n,\B_n,m)$ is empty. We choose such a $\beta$ sufficiently small, set $C = 3/\beta^7$, and choose $\eps = \eps(\beta) > 0$ to be sufficiently small. Finally, let $n$ be sufficiently large. 

Applying the Main Algorithm to each independent set $I \in \SF^{(\delta)}_{\ge}(\HH_n,\B_n,m)$, that is, to every independent set in $\HH_n$ such that $|I| =m$ and $|I \setminus B| \ge \delta m$ for every $B \in \B_n$, we obtain (for each such $I$) a pair $(A,S)$ such that $S \subseteq I \subseteq A \cup S$. There are two cases to deal with, corresponding to the two possibilities that can occur at the end of the Main Algorithm. We first show that if $|A \setminus B| \le \gamma n$ for some $B \in \B_n$ then a much stronger bound holds.

\begin{claim}\label{MA4}
  There are at most $\delta^{m} |\B_n| {\|\B_n\| \choose m}$ sets $I \in \SF^{(\delta)}_{\ge}(\HH_n,\B_n,m)$ such that the Main Algorithm ends because $|A \setminus B| \le \gamma n$ for some $B \in \B_n$.
\end{claim}

\begin{proof}
We claim first that the number of such sets $I$ is at most
\begin{equation}\label{MAeq1}
  \sum_{B \in \B_n} \sum_{t = 0}^{\beta^2 m} \sum_{r \ge \delta m} \binom{n}{t} \binom{\gamma n}{r - t} \binom{\| \B_n \|}{m-r}.
\end{equation}
Indeed, let $S$ and $A$ be the selected and available sets at the end of the algorithm, set $t = |S|$, and recall that $t \le \beta^2 m$ by Claim~\ref{MA2}. We have at most ${n \choose t}$ choices for $S$ and, by Claim~\ref{MA3}, the set $S$ determines the set $A$. Let $B \in \B_n$ be such that $|A \setminus B| \le \gamma n$ and recall that $|I \setminus B| \ge \delta m$ by our assumption on $I$. Thus we must choose the set $B \in \B_n$, at least $\delta m - t$ elements of $A \setminus B$, and the remaining elements from $B$.

Note that $t \le \delta m/2$ by our choice of $\beta$ and so either $m \le (2\gamma/\delta)n$ or the number of choices for $I$ is zero. Since $\| B_n \| \ge \alpha n$ and $\gamma$ is small, it follows that the summand in~\eqref{MAeq1} is maximized exactly when $r = \delta m$. Now, using the inequalities ${n \choose k} \le \big( \frac{en}{k} \big)^k$ and
\begin{equation}\label{MAeq2}
{a \choose b - c} \, \le \, \left( \frac{b}{a-b} \right)^c {a \choose b},
\end{equation}
which holds for every $a > b > c \ge 0$, and since $t \le \delta m / 2$, $m \le (2\gamma/\delta)n \le (\alpha/2)n$ and $\| \B_n \| \ge \alpha n$, we can bound each summand in~\eqref{MAeq1} from above by
$$\left( \frac{en}{t} \right)^{t} \left( \frac{2e\gamma n}{\delta m} \right)^{\delta m - t}  \left( \frac{m}{\alpha n - m} \right)^{\delta m}  {\| \B_n \| \choose m}.$$   
Since $t \le \delta m / 2$ and $t \mapsto (c/t)^t$ is increasing on $(0,c/e)$, this is at most
$$\left( \frac{\delta m}{2 \gamma t} \right)^{t} \left( \frac{2e\gamma n}{\delta m} \cdot \frac{2m}{\alpha n} \right)^{\delta m}  {\| \B_n \| \choose m} \, \le \, \left( \frac{1}{\gamma} \right)^{\delta m/2} \left( \frac{4e\gamma}{\alpha \delta} \right)^{\delta m}  {\| \B_n \| \choose m} \, \le \, \delta^{2m} {\| \B_n \| \choose m}$$
if $\gamma \le (\alpha \delta / 4e)^2 \cdot \delta^{4/\delta}$. Since $m^2 \cdot \delta^{2m} \le \delta^m$, the claim follows.
\end{proof}

Finally, we deal with the case in which $|A| \le (\alpha - \beta)n$ for some $B \in \B_n$.

\begin{claim}\label{MA5}
There are at most $2^{-\eps m} {\|\B_n\| \choose m}$ sets $I \in \SF^{(\delta)}_{\ge}(\HH_n,\B_n,m)$ such that the Main Algorithm ends because $|A| \le (\alpha - \beta)n$.
\end{claim}

\begin{proof}
As in the previous claim, we have $t = |S| \le \beta^2 m$, by Claim~\ref{MA2}, and the set $S$ determines the set $A$, by Claim~\ref{MA3}. Thus, the number of choices for $I$ is at most
\begin{equation}\label{MAeq5}
  \sum_{t=0}^{\beta^2 m} {n \choose t} {(\alpha - \beta)n \choose m - t}.
\end{equation}
Now, recall that either $m \le \left( \alpha - 2\beta  \right) n$ or $\SF^{(\delta)}_{\ge}(\HH_n,\B_n,m)$ is empty. Thus, estimating each summand in~\eqref{MAeq5} as in the proof of Claim~\ref{MA4}, we obtain
$${n \choose t} {(\alpha - \beta)n \choose m - t} \, \le \, \bigg( \frac{en}{t} \bigg)^t \left( \frac{m}{(\alpha - \beta)n - m} \right)^t {(\alpha - \beta)n \choose m} \, \le \, \left( \frac{e m}{ \beta t} \right)^t {(\alpha - \beta)n \choose m}.$$
Now, using the inequality $\binom{b}{c} \le \left( \frac{b}{a} \right)^c \binom{a}{c}$, which is valid for all $a > b > c \ge 0$, and recalling that $t \le \beta^2 m$ and that $t \mapsto (c/t)^t$ is increasing on $(0,c/e)$, we get 
$$\left( \frac{e m}{ \beta t} \right)^t {(\alpha - \beta)n \choose m} \, \le \, \left( \frac{e}{ \beta^3} \right)^{\beta^2 m} \left( \frac{\alpha - \beta}{\alpha} \right)^m {\alpha n \choose m} \, \le \, \left( \frac{e}{ \beta^3} \right)^{\beta^2 m} e^{- \beta m / \alpha} {\alpha n \choose m}.$$
Since $\| B_n \| \ge \alpha n$, the right-hand side is at most $\frac{1}{m} \cdot 2^{-\eps m} { \| \B_n \| \choose m}$ if $\beta > 0$ and $\eps = \eps(\beta) > 0$ are sufficiently small, as required. 
\end{proof}

Combining Claims~\ref{MA4} and~\ref{MA5}, we obtain Theorem~\ref{algprop}.
\end{proof}

\section{Janson argument}\label{JansonSec}

In this section, we shall complete the proof of Theorem~\ref{genthm} by showing that, under certain conditions, almost all independent (i.e., sum-free) sets $I$ of size $m$ in $\HH_n$ either satisfy $I \subseteq B$ for some $B \in \B_n$, or $|I \setminus B| \ge \delta m$ for every $B \in \B_n$. The key properties of $\HH$ which we will use are that $\Delta_2(\HH_n) = O(1)$, and that $\delta(\HH_n,\B_n) = \Omega(n)$; our key tool will be Janson's inequality. An argument similar to that presented in this section was used in~\cite{BMS} to study sum-free sets in random subsets of Abelian groups.

Given a hypergraph $\HH_n$, a family of sets $\B_n$ and $\delta > 0$, we define
$$\SF^{(\delta)}_{\le}(\HH_n,\B_n,m) \, := \, \Big\{ I \in \SF(\HH_n,m) \,\colon\, |I \setminus B| \le \delta m \textup{ for some } B \in \B_n \Big\}.$$
The following proposition shows that, if $\delta > 0$ is sufficiently small, then almost all independent sets in $\SF^{(\delta)}_{\le}(\HH_n,\B_n,m)$ are contained in some $B \in \B_n$. We write $2^B$ to denote the power set of $B$, i.e., the family of all subsets of $B$.

\begin{prop}\label{prop:Janson}
Let $\alpha > 0$, let $\HH = (\HH_n)_{n \in \N}$ be a sequence of $3$-uniform hypergraphs with $\Delta_2(\HH_n) = O(1)$, and let $\B = (\B_n)_{n \in \N}$ be a family of sets with $\| \B_n \| \ge \alpha n$. For every $\beta > 0$, there exists constants $\delta > 0$ and $C_0 > 0$ such that the following holds. If $\delta(\HH_n,\B_n) \ge \beta n$ and $C \ge C_0$, then
$$\Big| \SF^{(\delta)}_{\le}(\HH_n,\B_n,m) \setminus \bigcup_{B \in \B_n} 2^B \Big| \, \le \, n^{-C} |\B_n| {\|\B_n\| \choose m}$$
for every $m \ge C\sqrt{n \log n}$.
\end{prop}

Note that $\delta$ and $C_0$ in the statement of Proposition~\ref{prop:Janson} may depend on $\HH$, $\B$, $\alpha$ and $\beta$. We begin by recalling the Janson inequalities, and some basic facts about the hypergeometric distribution.

\subsection{The hypergeometric distribution}

The following well-known inequality (see~\cite[page 35]{JLR}, for example) allows us to deduce bounds in the hypergeometric distribution from results on product measure. For completeness we give a proof.

\begin{lemma}[Pittel's inequality]\label{pittel}
Let $m,n \in \N$, and set $p = m/n$. For any property $\Q$ on $[n]$ we have
$$\Pr\big( \Q \text{ holds for a random $m$-set} \big) \, \le \, 3\sqrt{m} \cdot  \Pr\big( \Q \text{ holds for a random $p$-subset of $[n]$} \big).$$
Moreover, if $\Q$ is monotone decreasing and $m \le n - 1$, then
$$\Pr\big( \Q \text{ holds for a random $m$-set} \big) \, \le \, C \cdot  \Pr\big( \Q \text{ holds for a random $p$-subset of $[n]$} \big)$$
for some absolute constant $C > 0$.
\end{lemma}

\begin{proof}
For the first part, simply note that a random $p$-subset of $[n]$ has size $m = pn$ with probability at least $1/(3\sqrt{m})$. If $\Q$ is monotone decreasing, say, then we apply the `Local LYM inequality' to $\Q_m$, the set of $m$-sets in $\Q$, and deduce that
$$\Pr\big( \Q \text{ holds for a random $k$-set} \big) \ge \, \Pr\big( \Q \text{ holds for a random $m$-set} \big)$$
for every $k \le m$. It is well-known that the median of the binomial distribution lies between $\lfloor pn \rfloor$ and $\lceil pn \rceil$, and if $m \le n - 1$ then it is easy to see that $\Bin(n,p) = \lceil pn \rceil$ has probability at most $(1 - 1/n)^{n-1} \to 1/e$ as $n \to \infty$. Thus, if $m \le n - 1$ and $n$ is sufficiently large, then a random $p$-subset of $[n]$ has size at most $m = pn$ with probability at least $1/2 - 1/e + o(1)$ as $n \to \infty$, and the result follows.
\end{proof}

The following result is an easy corollary of Janson's inequality (see~\cite{AS,JLR}), combined with Pittel's inequality.

\begin{lemma}[Hypergeometric Janson Inequality]\label{HJI}
Suppose that $\{U_i\}_{i \in I}$ is a family of subsets of an $n$-element set $X$ and let $m \in \{0, \ldots, n\}$. Let
$$\mu = \sum_{i \in I} (m/n)^{|U_i|} \quad \text{and} \quad \Delta = \sum_{i \sim j} (m/n)^{|U_i \cup U_j|},$$
where the second sum is over ordered pairs $(i,j)$ such that $i \neq j$ and $U_i \cap U_j \neq \emptyset$. Let $R$ be a uniformly chosen random $m$-subset of $X$. Then 
$$\Pr\big( U_i \nsubseteq R \text{ for all $i \in I$} \big) \,\le\, C \cdot \max\left\{e^{-\mu/2}, e^{-\mu^2/(2\Delta)}\right\},$$
where $C > 0$ is the constant in Pittel's inequality.
\end{lemma}

We now return to the proof of Proposition~\ref{prop:Janson}.

\subsection{Proof of Proposition~\ref{prop:Janson}}

We begin by partitioning $\SF^{(\delta)}_{\le}(\HH_n,\B_n,m)$ according to the set $B \in \B_n$ such that $|I \setminus B| \le \delta m$, and also according to the set $S = I \setminus B$. (Technically there could be more than one such set $B$, so in fact this might be a cover, rather than a partition.) Set
$$I(B,S) \, := \, \Big| \Big\{ I \in \SF(\HH_n,m) \,\colon\, I \setminus B = S \Big\} \Big|.$$ 
We shall prove the following lemma.

\begin{lemma}\label{lem:SF1}
Let $\HH = (\HH_n)_{n \in \N}$ be a sequence of $3$-uniform hypergraphs with $\Delta_2(\HH_n) = O(1)$. For every sufficiently small $\beta > 0$, there exists a $C_0 > 0$ such that the following holds. Let $C \ge C_0$, let $B \subseteq [n]$ with $\delta(\HH_n,B) \ge \beta n$, and let $S \subseteq [n] \setminus B$. Then, writing $k = |S|$, 
$$I(B,S) \; \le \; \Big( |B|^{-5Ck} + e^{-\beta^3 m} \Big) {|B| \choose {m-k}}$$
for every $m \ge C \sqrt{n \log n}$. 
\end{lemma}

In order to prove Lemma~\ref{lem:SF1}, we shall apply the following lemma to the Cayley graph of $S$, restricted to $B$. The lemma is a straightforward consequence of the Hypergeometric Janson's inequality. 

\begin{lemma}\label{lem:Zhao}
For every $\beta > 0$, there exists a constant $C_0 > 0$ such that the following holds. Let $\G$ be a graph on $n$ vertices with maximum degree at most $d$. If
$$e(\G) \, \ge \, 4\beta d n,$$
then for every $C \ge C_0$,
$$I(\G,m) \, \le \, \Big( n^{-Cd} + e^{-\beta m} \Big) {n \choose m}$$
for every $m \ge C \sqrt{n \log n}$.
\end{lemma}

\begin{proof}
Let $\{U_i\}_{i \in I}$ be the collection of pairs of vertices which span an edge of $\G$, so $U_i \nsubseteq R$ for all $i \in I$ if and only if $R$ is an independent set in $\G$. It is easy to see that, letting $\mu$ and $\Delta$ to be the quantities defined in the statement of Lemma~\ref{HJI},
$$\mu \,=\, e(\G) \frac{m^2}{n^2}  \quad \text{and} \quad \Delta \,\le\, {d \choose 2} \left( \frac{2 e(\G)}{d} \right) \bigg( \frac{m}{n} \bigg)^3 \le \, e(\G) \frac{d m^3}{n^3}.$$
Thus, by our bounds on $e(\G)$ and $m$, and assuming $C \ge 4/\beta$, 
$$\mu \, \ge \, 4 C d \log n  \quad \text{and} \quad \frac{\mu^2}{\Delta} \,\ge\, e(\G) \frac{m}{dn} \,\ge\, 4\beta m.$$

By the Hypergeometric Janson Inequality, 
\begin{eqnarray*}
I(\G,m) / {n \choose m} & \le & C \cdot \max \Big\{e^{-\mu/2}, e^{-\mu^2/(2\Delta)} \Big\} \\
& \le & C \cdot \max\Big\{n^{-2Cd}, e^{-2\beta m}\Big\} \,\le\, \max\Big\{ n^{-Cd}, e^{-\beta m} \Big\},
\end{eqnarray*}
as required.
\end{proof}

Recall that, given $\HH_n$, the Cayley graph $\G_S$ of $S$ is defined to be the graph with vertex set $V(\HH_n)$ and edge set 
$$E(\G_S) \,=\, \Big\{ \{u,v\} \subseteq V(\HH_n) \,\colon\, \{u,v,w\} \in \HH_n \textup{ for some } w \in S \Big\}.$$ 
In order to apply Lemma~\ref{lem:Zhao}, we shall need the following easy property of the Cayley graph.

\begin{obs}\label{lem:dreg}
$\Delta(\G_S) \le |S| \Delta_2(\HH_n)$.
\end{obs}

We can now easily deduce Lemma~\ref{lem:SF1} from Lemma~\ref{lem:Zhao} and Observation~\ref{lem:dreg}. 

\begin{proof}[Proof of Lemma~\ref{lem:SF1}]
If $I$ is an independent set in $\HH_n$ containing $S$, then $I \setminus S$ is an independent set in $\G_S$, so
$$I\big( B,S \big) \; \le \; I\big( \G_S[B], m - k \big),$$
where $k = |S|$. Choose $\beta > 0$ sufficiently small so that $\Delta_2(\HH_n) \le 1/(2\beta)$, recall that $\delta(\HH_n,B) \ge \beta n$, note that $d = \Delta(\G_S) \le |S|\Delta_2(\HH_n) \le |S|/(2\beta)$, and observe that therefore
$$e\big(\G_S[B] \big) \, \ge \, \frac{\beta |S| n}{\Delta_2(\HH_n)} \, \ge \, 2\beta^2 |S| \cdot |B| \ge 4\beta^3d|B|.$$
Thus, by Lemma~\ref{lem:Zhao}, if $\beta < 1/10$ then $d \ge 5|S| = 5k$ and
$$I\big( \G_S[B], m - k \big) \, \le \, \Big( |B|^{-5Ck} + e^{-\beta^3 m} \Big) {|B| \choose {m-k}},$$
for every $m \ge C \sqrt{n \log n}$, as required. 
\end{proof}

Finally, let us deduce Proposition~\ref{prop:Janson} from Lemma~\ref{lem:SF1}.

\begin{proof}[Proof of Proposition~\ref{prop:Janson}]
Summing over all sets $B \in \B_n$ and subsets $S \subseteq [n] \setminus B$, and applying Lemma~\ref{lem:SF1}, we have
\begin{eqnarray*}
\Big| \SF^{(\delta)}_{\le}(\HH_n,\B_n,m) \setminus \bigcup_{B \in \B_n} 2^B \Big| & \le & \sum_{B \in \B_n} \sum_{k = 1}^{\delta m} \sum_{S \subseteq [n] \setminus B \,:\, |S| = k} I(B,S)\\
& \le & |\B_n| \sum_{k = 1}^{\delta m} {n \choose k} \Big( n^{-4(C+k)} + e^{-\beta^3 m} \Big) {\|\B_n\| \choose {m-k}},
\end{eqnarray*}
for every $m \ge C\sqrt{n \log n}$, since $\delta(\HH_n,\B_n) \ge \beta n$ and $\Delta_2(\HH_n) = O(1)$ together imply that $|B| = \Theta(n)$ for every $B \in \B_n$. We consider three cases.

\medskip
\noindent \textbf{Case 1:} If $n^{-4(C+k)} \ge e^{-\beta^3 m}$, then  
\begin{eqnarray*}
{n \choose k} \Big( n^{-4(C+k)} + e^{-\beta^3 m} \Big) {\|\B_n\| \choose {m-k}} \, \le \, n^{-2C} {\|\B_n\| \choose {m}},
\end{eqnarray*}
since ${\|\B_n\| \choose {m-k}} \, \le \, {n \choose k} {\|\B_n\| \choose {m}}$.

\medskip
\noindent \textbf{Case 2:} If $n^{-4(C+k)} \le e^{-\beta^3 m}$ and $m \le \alpha n / 2$, then by~\eqref{MAeq2} we have
$${\|\B_n\| \choose {m-k}} \, \le \, \left( \frac{m}{ \alpha n - m} \right)^k {\|\B_n\| \choose {m}} \, \le \, \left( \frac{2m}{ \alpha n} \right)^k {\|\B_n\| \choose {m}},$$
since $\| \B_n \| \ge \alpha n$. Thus, using the bound ${n \choose k} \le \big( \frac{en}{k} \big)^k$, we have
\begin{eqnarray*}
{n \choose k} \Big( n^{-4(C+k)} + e^{-\beta^3 m} \Big) {\|\B_n\| \choose {m-k}} \, \le \, 2 \cdot e^{-\beta^3 m} \left( \frac{2em}{ \alpha k} \right)^k {\|\B_n\| \choose {m}} \, \le \, e^{-\beta^3 m / 2} {\|\B_n\| \choose {m}},
\end{eqnarray*}
if $\delta = \delta(\alpha,\beta) > 0$ is sufficiently small, since $k \le \delta m$.

\medskip
\noindent \textbf{Case 3:} If $\| \B_n \|^{-4(C+k)} \le e^{-\beta^3 m}$ and $m \ge \alpha n / 2$, then we again use the (trivial) bound ${\|\B_n\| \choose {m-k}} \, \le \, {n \choose k} {\|\B_n\| \choose {m}}$, to obtain
$${n \choose k} \Big( n^{-4(C+k)} + e^{-\beta^3 m} \Big) {\|\B_n\| \choose {m-k}} \, \le \, 2e^{-\beta^3 m} {n \choose k}^2 {\|\B_n\| \choose {m}} \, \le \, e^{-\beta^3 m / 2} {\|\B_n\| \choose {m}},$$
if $\delta = \delta(\alpha,\beta) > 0$ is sufficiently small, since ${n \choose k} \le \binom{2m/\alpha}{k} \le \left(\frac{2e}{\alpha\delta}\right)^{\delta m} \le e^{-\beta^3 m/6}$ for $k \le \delta m$.

\medskip
Since $e^{-\beta^3 m/2} \ll n^{-2C}$ for $m \ge C\sqrt{n \log n}$, the claimed bound follows. 
\end{proof}

We finish this section by observing that Theorem~\ref{algprop} and Proposition~\ref{prop:Janson} together imply Theorem~\ref{genthm}.

\begin{proof}[Proof of Theorem~\ref{genthm}]
Let $\HH = (\HH_n)_{n \in \N}$ be a sequence of $3$-uniform hypergraphs which is $(\alpha,\B)$-stable, where $\B = (\B_n)_{n \in \N}$ is a family of sets, and $\alpha > 0$. Suppose that $\alpha(\HH_n) \ge \|\B_n\| \ge \alpha n$, and that there exists $\beta > 0$ such that $e(\HH_n) \ge \beta n^2$, $\Delta_2(\HH_n) \le 1/\beta$, $|\B_n| \le n^{1/\beta}$ and $\delta(\HH_n,\B_n) \ge \beta n$ for every $n \in \N$. Let $\delta = \delta(\beta) > 0$ be sufficiently small, and let $C = C(\beta, \delta) > 0$ be sufficiently large. We claim that if
$$m \,\ge\, C\sqrt{n \log n},$$
then almost every independent set in $\HH_n$ of size $m$ is a subset of some $B \in \B_n$.

Indeed, by Theorem~\ref{algprop}, the number of independent sets $I$ in $\HH_n$ of size $m$ for which $|I \setminus B| \ge \delta m$ for every $B \in \B_n$ is at most 
$$\Big( 2^{-\eps m} + \delta^m |\B_n| \Big) {\|\B_n\| \choose m}$$
for some $\eps > 0$ and by Proposition~\ref{prop:Janson}, the number of such sets for which $1 \le |I \setminus B| \le \delta m$ for some $B \in \B_n$ is at most 
$$n^{-C} |\B_n| {\|\B_n\| \choose m}.$$
Since $|\B_n| \le n^{1/\beta}$, $C > 1/\beta$, and $\alpha(\HH_n) \ge \|\B_n\|$, the result follows.
\end{proof}

\section{Abelian groups of Type I}\label{GroupSec}

In this section, we shall use Theorem~\ref{genthm} to prove Theorem~\ref{thm:groups} for all $q > 2$. We remark that the  proof below can also be adapted to cover the case $q = 2$; however, since we shall give a different proof of the case $q = 2$ in Section~\ref{q=2Sec}, we leave the details to the reader. (If $\HH_n$ denotes the hypergraph that encodes Schur triples in a group $G$ of even order $n$ and $\B_n$ denotes the collection of maximum-size sum-free subsets of $G$, then it is not always true that $\Omega(\HH_n, \B_n) = \Omega(n)$. This problem can be easily overcome by considering triples of the form $(x,x,2x)$, cf.~the proof of the $1$-statement in \cite[Theorem~1.2]{BMS}.)

In order to prove that our hypergraph is $(\alpha,\B)$-stable, we shall use the following result (see~\cite[Corollary~2.8]{BMS}), which follows immediately by combining results of Green and Ruzsa~\cite{GR05} and Lev, {\L}uczak, and Schoen~\cite{LLS}. Let $\SF_0(G)$ denote the collection of maximal-size sum-free subsets of $G$ and recall that each $B \in \SF_0(G)$ has size $\mu(G) |G|$. 

\begin{prop}\label{GRLLS}
Let $G$ be a finite Abelian group of Type $I(q)$, where $q \equiv 2 \pmod 3$ and let $0 < \gamma < \gamma(q)$ and $0 < \beta < \beta_0(\gamma,q)$ be sufficiently small. Let $A \subseteq G$, and suppose that 
$$|A| \, \ge \, \big( \mu(G) - \beta \big) |G|.$$ Then one of the following holds: 
\begin{itemize}
\item[$(a)$] $|A \setminus B| \le \gamma |G|$ for some $B \in \SF_0(G)$.\smallskip
\item[$(b)$] $A$ contains at least $\beta |G|^2$ Schur triples.
\end{itemize}
\end{prop}

We shall also use the following classification of extremal sum-free sets for Type I groups.

\begin{thm}[Diananda and Yap~\cite{DY}]\label{thm:SFG-structure}
Let $G$ be a finite Abelian group of Type $I(q)$, where $q \equiv 2 \pmod 3$. Then every $B \in \SF_0(G)$ is a union of cosets of some subgroup $H$ of $G$ of index $q$, $B/H$ is an arithmetic progression in $G/H$, and $B \cup (B+B) = G$. 

In other words, for every $B \in \SF_0(G)$, there exists a homomorphism $\varphi \colon G \to \ZZ_q$ such that $B = \varphi^{-1}(\{k+1, \ldots, 2k+1\})$, where $q = 3k + 2$. 
\end{thm}

Combining Theorem~\ref{thm:SFG-structure} with Kronecker's Decomposition Theorem, we easily obtain the following well-known corollary.


\begin{cor}\label{cor:SFG}
Let $G$ be an arbitrary group of Type I. Then $|\SF_0(G)| \le |G|$.
\end{cor}

It is now straightforward to deduce Theorem~\ref{thm:groups} from Theorem~\ref{genthm}, Proposition~\ref{GRLLS}, and Corollary~\ref{cor:SFG}.

\begin{proof}[Proof of Theorem~\ref{thm:groups} for $q \neq 2$]
Let $q \equiv 2 \pmod 3$ be an odd prime, let $C = C(q)$ be sufficiently large, and let $G_n$ be an Abelian group of Type I($q$), with $|G_n| = n$. We shall show that if $m \ge C(q)\sqrt{n \log n}$, then almost every sum-free set of size $m$ in $G_n$ is contained in a member of $\SF_0(G)$.

We begin by choosing an infinite set $X \subseteq \N$ such that, for every $n \in X$, $q$ is the smallest prime divisor of $n$ with $q \equiv 2 \pmod 3$. For each $n \in X$, let $G_n$ be an Abelian group of Type I($q$), with $|G_n| = n$, and define $\HH = (\HH_n)_{n \in X}$ to be the sequence of hypergraphs on vertex set $V(\HH_n) = G_n$ which encodes Schur triples. To be precise, let $V(\HH_n) = G_n$, let $\{x,y,z\} \in {G_n \choose 3}$ be an edge of $\HH_n$ whenever $x + y = z$, and observe that every sum-free subset of $G_n$ is an independent set in $\HH_n$.\footnote{But not vice-versa, since $\HH_n$ does not contain the Schur triples in $G_n$ of the form $(x,x,2x)$. Thus, by bounding $I(\HH_n,m)$ we are in fact proving a statement which is slightly stronger than Theorem~\ref{thm:groups}.} Let $\B_n = \SF_0(G_n)$, the collection of maximum size sum-free subsets of $G_n$, and recall that $|\B_n| \le n$, by Corollary~\ref{cor:SFG}. 

We claim that $\HH$ and $\B$ satisfy the conditions of Theorem~\ref{genthm}. Indeed, $\HH_n$ is 3-uniform, has $\Theta(n^2)$ edges, and satisfies $\Delta_2(\HH_n) = 3$. Setting $\alpha = \mu(G)$, we have $\alpha(\HH_n) = \|\B_n\| = \alpha n$ and $|\B_n| \le n$, as observed above. Moreover, the statement that $\HH$ is $(\alpha,\B)$-stable is exactly Proposition~\ref{GRLLS}. Thus it will suffice to show that $\delta(\HH_n,\B_n) = \Omega(n)$.

\begin{claim*}
For each $B \in \SF_0(G)$ and every $x \in G \setminus B$, 
$$\bigg| \bigg\{ \big\{ y, z \big\} \in {B \choose 2} \colon x = y + z \bigg\} \bigg| \, \ge \, \frac{n}{2q} - \frac{1}{2}.$$
\end{claim*}

\begin{proof}[Proof of claim]
Let $B \in \SF_0(G)$ and let $x \in G \setminus B$. By Theorem~\ref{thm:SFG-structure}, there exists a subgroup $H$ of $G$ of index $q$ such that $B$ is a union of cosets of $H$ and $B \cup (B+B) = G$. It follows that $x = y + z$ for some $y, z \in B$, and that $y+h, z-h \in B$ for every $h \in H$. Thus, 
$$\big\{ y + h,z - h \big\} \in C(x) := \bigg\{ \big\{ y, z \big\} \in {B \choose 2} \colon x = y + z \bigg\}$$ 
whenever $h \in H$ and $y + h \neq z - h$. Moreover, since $|G|$ is odd, there is at most one $h \in H$ such that $2h = z - y$, so $|C(x)| \ge (|H| - 1)/2 = n/2q - 1/2$, as required.
\end{proof}

Thus the pair $(\HH,\B)$ satisfies the conditions of Theorem~\ref{genthm} and hence if $C(q)$ is sufficiently large and $m \ge C(q) \sqrt{n \log n}$, then almost every sum-free set of size $m$ in $G_n$ is contained in some $B \in \B_n$, as required.

Finally, let us deduce that if $G$ is an Abelian group of Type I($q$), and $m \ge C(q) \sqrt{n \log n}$, then  
\[
|\SF(G,m)| \,=\, \frac{1}{2} \cdot \left(\#\big\{\text{elements of $G$ of order $q$}\big\} + o(1)\right) {\mu(G)n \choose m}.
\]
Indeed, it suffices to observe that \(|\SF_0(G)| = \#\big\{\text{elements of $G$ of order $q$}\big\} / 2\), by Theorem~\ref{thm:SFG-structure}, and that each pair $B,B' \in \SF_0(G)$ intersect in at most $(1 - 1/q)\mu(G) |G|$ elements. The result now follows from some easy counting. 
\end{proof}

\section{Abelian groups of even order}\label{q=2Sec}

In this section, we shall prove the following theorem, which implies Theorem~\ref{thm:groups} in the case $q = 2$. We shall use Theorem~\ref{thm:graphs} and some ideas from Section~\ref{JansonSec}, but otherwise this section is self-contained. In particular, we shall not use Proposition~\ref{GRLLS} and thus we give a new proof of the main theorem of~\cite{LLS} and~\cite{Sap02}.

\begin{thm}\label{q=2}
  If $G$ is an Abelian group of order $n$, then 
  $$|\SF(G,m)| \; = \; \left(\# \big\{ \text{elements of $G$ of order $2$} \big\} + o(1) \right) {n/2 \choose m}$$ 
  for every $m \ge 4\sqrt{n \log n}$. 
\end{thm}

We remark that we shall prove the theorem for \emph{all} finite Abelian groups, not just those of even order. We begin by partitioning the collection of sum-free sets into two pieces. Given an Abelian group $G$, let
$$\SF_\le^{(\delta)}(G,m) \, := \, \Big\{ I \in \SF(G,m) \,\colon\, |I \cap H| \le \delta m \textup{ for some } H \le G \textup{ with } [G:H] = 2 \Big\},$$
and 
$$\SF_\ge^{(\delta)}(G,m) \, := \, \Big\{ I \in \SF(G,m) \,\colon\, |I \cap H| \ge \delta m \textup{ for every } H \le G \textup{ with } [G:H] = 2 \Big\}.$$
Note that if $|G|$ is odd then $\SF_\le^{(\delta)}(G,m)$ is empty. We shall prove the following proposition using the method of Section~\ref{JansonSec}.

\begin{prop}\label{prop:SF1}
Let $G$ be an Abelian group of order $n$, and let $\delta > 0$ be sufficiently small. Then
\[
|\SF_\le^{(\delta)}(G,m)| \, \le \, \left(\#\big\{ \text{elements of $G$ of order $2$} \big\} +o(1) \right) \cdot {{n/2} \choose m}
\]
for every $m \ge 4 \sqrt{n \log n}$. 
\end{prop}

For sets in $\SF_\ge^{(\delta)}(G,m)$, i.e., far from any $H \le G$ of index $2$, we shall prove the following stronger bound using Theorem~\ref{thm:graphs}.

\begin{prop}\label{prop:SF2}
  Let $G$ be an Abelian group of order $n$, and let $\delta > 0$. If $\eps = \eps(\delta) > 0$ is sufficiently small and $C = C(\delta)$ is sufficiently large, then
  $$|\SF_\ge^{(\delta)}(G,m)| \; \le \; 2^{-\eps m} {{n/2} \choose m}$$
  for every $m \ge C\sqrt{n}$ and every sufficiently large $n \in \N$.
\end{prop}

We begin by proving Proposition~\ref{prop:SF1}. In this section, we shall use a slightly different notion of Cayley graph than that used earlier. Given $S \subseteq G$, define $\G^*_S$ to be the graph with vertex set $G \setminus S$ and edge set $\big\{  xy \colon x - y \in S \big\}$, and note that if $I$ is a sum-free set in $G$ with $S \subseteq I$, then $I \setminus S$ is an independent set in $\G^*_S$. 

\begin{proof}[Proof of Proposition~\ref{prop:SF1}] 
  Let $G$ be an Abelian group of even order $n$, let $H$ be a subgroup of $G$ of index $2$, and let $S \subseteq H$ satisfy $|S| = k \le \delta m$. Set $\gamma = 1/65$. We claim that for every $m \ge 4 \sqrt{n \log n}$, there are at most
\begin{equation}\label{eq:even1}
\Big( n^{-4k} + e^{-\gamma m} \Big) {n/2 \choose {m-k}}
\end{equation}
sum-free subsets $I$ of $G$ of order $m$ with $I \cap H = S$.

Observe first that the graph $\G^*_S[G \setminus H]$ is $d$-regular, where $d = |S \cup (-S)| \in [k,2k]$. Indeed, for each $x \in G \setminus H$, let
$$N(x) \, = \, \big\{ y \in G \setminus H \,\colon\, x - y \in S \textup{ or } y - x \in S \big\}.$$
Since $S \subseteq H$, it follows that $x - S$ and $x + S$ are in $G \setminus H$, and hence $|N(x)| = |S \cup (-S)|$, as claimed. Since $|S| = k$, we have $k \le d \le 2k$.

Now, by the Hypergeometric Janson Inequality, Lemma~\ref{HJI}, there are at most  
$$C \cdot \max\Big\{e^{-km^2/4n}, e^{-m/64} \Big\} {n/2 \choose {m-k}} \,\le\, \Big( n^{-4k} + e^{-\gamma m} \Big) {n/2 \choose {m-k}},$$
independent sets of size $m - k$ in $\G^*_S[G \setminus H]$. This follows because $k \le \delta m$, so
$$\mu \, \ge \, \left( \frac{kn}{4} \right) \left( \frac{(m-k)^2}{(n/2)^2} \right) \, \ge \, \frac{km^2}{2n} \quad \text{and} \quad \Delta \,\le\, {2k \choose 2} \frac{n}{2} \left( \frac{(m-k)^3}{(n/2)^3} \right) \, \le \, \frac{8k^2m^3}{n^2},$$
and $m \ge 4 \sqrt{n \log n}$. Since each sum-free subset $I \subseteq G$ induces an independent set in $\G^*_S[G \setminus H]$, then~\eqref{eq:even1} follows.

Finally, summing~\eqref{eq:even1} over subgroups $H$ and sets $S$, we obtain
\begin{eqnarray}
  \label{eq:SF-le-delta}
  |\SF_\le^{(\delta)}(G,m)| & \le & \#\big\{H \le G \colon [G:H] = 2 \} \sum_{k = 0}^{\delta m} {n \choose k} \Big( n^{-4k} + e^{-\gamma m} \Big) {n/2 \choose {m-k}}\\
  \nonumber
  & \le & \left( \# \big\{\text{elements of $G$ of order $2$} \big\} \,+\, O\left( \frac{1}{n^2} \right) \right) {n/2 \choose m} 
\end{eqnarray}
for every $m \ge 4\sqrt{n \log n}$. To see the last inequality, observe that the number of subgroups $H$ of index $2$ in $G$ is exactly the number of elements of $G$ of order $2$ and consider three cases as in the proof of Proposition~\ref{prop:Janson}. Indeed, if $n^{-4k} \ge e^{-\gamma m}$ or $m \ge n/4$, then each summand in~\eqref{eq:SF-le-delta} is at most $\big(n^{-2k} + e^{-\gamma m/2} \big) {n/2 \choose m}$ by the trivial bound ${n/2 \choose m-k} \le {n \choose k} {n/2 \choose m}$. But if $n^{-4k} \le e^{-\gamma m}$ and $m \le n / 4$, then by~\eqref{MAeq2}, 
$${n \choose k}{n/2 \choose {m-k}} \, \le \, \left( \frac{en}{k} \right)^k \left( \frac{2m}{n - 2m} \right)^k {n/2 \choose m} \, \le \, \left( \frac{4em}{k} \right)^k {n/2 \choose m} \, \le \, e^{O(\sqrt{\delta} m)} {n/2 \choose m}$$
since $k \le \delta m$. Thus, if $\delta > 0$ is chosen small enough, then each summand in~\eqref{eq:SF-le-delta} is at most $e^{-\gamma m/2} {n/2 \choose m}$, as required. 
\end{proof}

We next turn to the proof of Proposition~\ref{prop:SF2}. We shall divide
into two cases: either the smallest eigenvalue $\lambda(I)$ of $I$
(see below) is at most $(\delta - 1)|I|$, in which case we shall use
some basic facts about characters of finite Abelian groups to show
that there are few such sets; or $\lambda(I)$ is larger, in which
case we shall find a small subset $S \subseteq I$ such that $\G^*_S$
is a $d$-regular graph with smallest eigenvalue satisfying
$\lambda > (\delta/4-1)d$, and apply
Theorem~\ref{thm:graphs}. We begin with the following key definition.

\begin{defn}[The smallest eigenvalue of $S$]
  Given a finite Abelian group $G$, and a subset $0 \not\in S \subseteq G$, let 
  \[
  \lambda(S) \, := \, \min\big\{ \Re(\lambda) \,\colon\, A(S)v = \lambda v \text{ for some $v \neq \mathbf{0}$} \big\},
  \]
  where $\Re(\lambda)$ is the real part of the complex number $\lambda$, $A(S)$ is the adjacency matrix of the directed Cayley graph on $G$, i.e., the $(0,1)$-matrix with $A(x,y) = 1$ iff $y - x \in S$, and $\mathbf{0}$ is the zero vector.
\end{defn}

Next, we recall some simple properties of characters of finite Abelian groups.

\subsection{Characters of finite Abelian groups}

\begin{defn}
  A {\em character} of a group $G$ is a homomorphism from $G$ into the multiplicative group of non-zero complex numbers, i.e., a function $\chi \colon G \to \Cs$ such that $\chi(a + b) = \chi(a)\chi(b)$ for all $a, b \in G$. 
\end{defn}

A character $\chi$ is called trivial if $\chi(x) = 1$ for all $x \in G$;
we will denote the trivial character by $\chit$. The set of all characters
of $G$ is denoted by $\Gh$. The following statement establishes a relation
between the smallest eigenvalue of the matrix $A(S)$ and the characters
of $G$.

\begin{lemma}\label{lambda}
  For every $0 \not\in S \subseteq G$, 
  \[
  \lambda(S) \,=\, \min\left\{ \Re \left( \sum_{s \in S} \chi(s) \right) \,\colon\, \chi \in \Gh \right\} \,\ge\, -|S|.
  \]
\end{lemma}

We shall use the following facts about finite Abelian groups in the proof of Lemma~\ref{lambda}.

\begin{fact}\label{fact1}
If $G$ is a finite Abelian group of order $n$, then all its characters take values in the set 
$$U_n \, := \, \big\{ (\xi_n)^k \colon k \in \{0, \ldots, n-1\} \big\}, \quad \text{where} \quad \xi_n = e^{2\pi i/n},$$
of $n$th roots of unity. Moreover, if $\chi \in \Gh$ then $\rng(\chi) = U_k$ for some $k = k(\chi)$.
\end{fact}

\begin{fact}\label{fact2}
$\Gh$ is an orthogonal basis of the vector space $\C^G$.
\end{fact}

\begin{proof}[Proof of Lemma~\ref{lambda}]
Let us start by breaking up the adjacency matrix $A(S)$ into $|S|$ pieces as follows:
$$A(S) = \sum_{s \in S} A_s, \quad \text{where} \quad A_s(x,y) = \left\{
\begin{array} {r@{\quad}ll} 1 & \text{if } \; y - x = s & \\[+1ex]
0 & \text{otherwise.} &
\end{array}\right.$$
Thus, for every $s, x \in G$ and $\chi \in \Gh$,
$$(A_s\chi)(x) \,=\, \chi(x+s) \,=\, \chi(x)\chi(s) \,=\, (\chi(s)\chi)(x),$$
and so every $\chi \in \Gh$ is an eigenvector of each $A_s$ with $\chi(s)$ being the corresponding eigenvalue. Hence, every $\chi \in \Gh$ is an eigenvector of $A(S)$, with eigenvalue $\sum_{s \in S} \chi(s)$. Since the characters of $G$ form an orthogonal basis of $\C^G$, it follows that the set of eigenvalues of $A(S)$ is exactly
$$\bigg\{ \sum_{s \in S} \chi(s) \,\colon\, \chi \in \Gh \bigg\}.$$
The inequality $\lambda(S) \ge -|S|$ follows since $|\chi(s)| = 1$ for every $s \in S$ and $\chi \in \Gh$.
\end{proof}

Let us note for future reference the following fact from the proof above.

\begin{lemma}\label{evectors}
  For every $0 \not\in S \subseteq G$, the characters of $G$ form a basis of eigenvectors of the matrix $A(S)$.
\end{lemma}

\subsection{Sum-free sets with small smallest eigenvalue}

Using the properties described above, we shall prove the following lemma.

\begin{lemma}\label{SF3_big}
  Let $\delta > 0$ be sufficiently small and let $n \in \N$ be sufficiently large. Then, there exist constants $\eps = \eps(\delta) > 0$ and $C = C(\delta)$ such that, for every $m \ge C\sqrt{n}$,
  \[
  \left| \Big\{ I \in \SF_\ge^{(\delta)}(G,m) \,\colon\, \lambda(I) \le (\delta - 1)|I| \Big\} \right| \, \le \, 2^{-\eps m} {{n/2} \choose {m}}.
  \]
\end{lemma}

\begin{proof}
  For each character $\chi$ of $G$ and each $A \subseteq G$, define $\lambda(A,\chi) = \Re \left( \sum_{x \in A} \chi(x) \right)$. We shall bound the number of $I \in \SF_\ge^{(\delta)}(G,m)$ such that $\lambda(I,\chi) \le (\delta-1)|I|$ by
  \[
  \frac{2^{-\eps m}}{n} \binom{n/2}{m}.
  \]
  The desired bound will follow since $\lambda(I) = \min_{\chi \in \Gh} \lambda(I,\chi)$ and there are at most $|G|$ characters of $G$. We split into two cases, depending on the number of different values taken by~$\chi$. 

  \bigskip
  \noindent {\bf Case 1.} $|\rng(\chi)| = 2$.
  \medskip
  
  Since $\chi$ is a group homomorphism, it corresponds to a subgroup $H$ of $G$ of index $2$, namely, $H = \chi^{-1}(1)$. Since $|I \cap H| \ge \delta n$ for every such $H$, we have
  \[
  \lambda(I,\chi) \,=\, \Re \left(\sum_{x \in I} \chi(x)\right)  \,=\, |I \cap H| - |I \setminus H| \,\ge\, \big( 2\delta - 1 \big) |I|
  \]
  and hence in this case, there are no $I \in \SF_{\ge}^{(\delta)}(G,m)$ such that $\lambda(I,\chi) \le (\delta - 1)|I|$.

  \bigskip
  \noindent
  {\bf Case 2.} $|\rng(\chi)| \ge 3$.
  \medskip

  Let $k = |\rng(\chi)|$ and recall (from Fact~\ref{fact1}) that $\rng(\chi) = U_k$, where $U_k$ is the multiplicative group of $k$th roots of unity. Observe that $| \chi^{-1}(\xi) | = n/k$ for every $\xi \in U_k$, and consider, for each $\zeta$ on the complex unit circle $S^1$, the open arc $C_ \zeta$ of length $\pi/3$ centred at $\zeta$ on $S^1$. Set $K_\zeta := \chi^{-1}(C_\zeta)$, and note that $|C_\zeta \cap U_k| \le k/3$ for every $\zeta \in S^1$, and hence $|K_\zeta| \le n/3$. Note also that, even though there are infinitely many $C_\zeta$, there are at most $2k$ different sets $K_\zeta$. 
  
  Let $c > 0$ and suppose first that there exists $\zeta \in S^1$ such that $|K_\zeta \cap I| \ge (1 - c)|I|$. The number of such sets $I$ is at most
  $$2k \cdot \sum_{\ell = 0}^{c m} {n - |K_\zeta| \choose \ell}{|K_ \zeta | \choose m - \ell} \,\le\, n^2 {n/3 \choose (1-c)m} {2n/3 \choose c m} \,\le\, \left(\frac{2}{3} + c'\right)^m {n/2 \choose m},$$
  where $c'(c) \to 0$ as $c \to 0$.

  So suppose that $|K_ \zeta \cap I| \le (1 - c)|I|$ for every $\zeta \in S^1$. We claim that, if $\delta > 0$ is sufficiently small, then
  \begin{equation}\label{eq:case2}
  |\lambda(I,\chi)| \,=\, \Big| \sum_{x \in I} \chi(x) \Big| \,\le\, \Big( 1 - c + c \cdot  \cos\big( \pi/6 \big) \Big) \cdot |I| \, < \, \big( 1 - \delta \big)|I|.
  \end{equation}
To see this let $v = \sum_{x \in I} \chi(x)$, note that if $v = 0$ then we are done, and otherwise observe that, by our assumption, $\chi(x)$ can lie within the open arc of length $\pi/3$ centred in direction $v$ for at most $(1-c)|I|$ elements $x \in I$. Since each of the others contribute at most $\cos(\pi /6)$ in the direction of $v$,~\eqref{eq:case2} follows. This is a contradiction, so the proof is now complete.
\end{proof}

\subsection{Sum-free sets with large smallest eigenvalue}

We shall prove the following statement using Theorem~\ref{thm:graphs}. Together with Lemma~\ref{SF3_big} it will easily imply Proposition~\ref{prop:SF2}, and hence Theorem~\ref{q=2}.

\begin{lemma}\label{SF3_small}
  For every finite Abelian group $G$ and every $\delta > 0$, there exist $\eps = \eps(\delta) > 0$ and $C = C(\delta) > 0$ such that 
  $$\left| \Big\{ I \in \SF_\ge^{(\delta)}(G,m) \,\colon\, \lambda(I) \ge (\delta - 1)|I| \Big\} \right| \; \le \; 2^{-\eps m} {{n/2} \choose {m}}$$
  for every $m \ge C\sqrt{n}$.
\end{lemma}

The idea of the proof is as follows: we choose a set $S \subseteq I$ of size $\eps m$ and observe that, since $I$ is sum-free, $I \setminus S$ is an independent set in $\G^*_S$, the Cayley graph of $S$. The key point is that, for some such $S$, our bound on $\lambda(I)$ implies the existence of a non-trivial bound on $\lambda(\G^*_S)$, the smallest eigenvalue of the adjacency matrix of the Cayley graph of $S$. Combined with Theorem~\ref{thm:graphs}, this implies that there are only very few choices for $I \setminus S$, and hence for $I$ itself.

The first step is the following lemma, which shows that our bound on $\lambda(I)$ allows us to find a small set $S$ such that $\lambda(S)/|S|$ is also bounded away from minus one.

\begin{lemma}\label{chooseS}
  If $I \in \SF_\ge^{(\delta)}(G,m)$ satisfies $\lambda(I) \ge (\delta - 1)|I|$, then there exists a set $S \subseteq I$ of size $\eps m$ such that
  \begin{equation}
    \label{eq:lambdaS}
    \lambda(S) \,\ge\, \left( \frac{\delta}{2} - 1 \right)|S|.
 \end{equation}
\end{lemma}
\begin{proof}
Recall the definition of $\lambda(A,\chi)$ from the proof of Lemma~\ref{SF3_big}. Since $\lambda(I) \ge (\delta - 1)|I|$, it follows from Lemma~\ref{lambda} that $\lambda(I,\chi) \ge (\delta - 1)|I|$ for every $\chi \in \Gh$. Choose a subset $S \subseteq I$ of size $\eps m$ uniformly at random; we claim that $\lambda(S,\chi)$ is tightly concentrated around the mean, i.e., around $\eps\lambda(I,\chi)$. Indeed, by Chernoff's inequality, we have
\[
\Pr\Big( \lambda(S,\chi) \le (\delta/2-1)\eps|I| \Big) \, \le \, e^{- \Omega(m)},
\]
where the implicit constant depends on $\eps$ and $\delta$. There are exactly $n$ characters in $\Gh$, and so, by the union bound, the probability that $S$ does not satisfy~\eqref{eq:lambdaS} is at most $1/2$. Thus there exists a set $S$ as claimed. 
\end{proof}

Next, we show that this bound on $\lambda(S)$ implies a similar bound on $\lambda(\G^*_S)$, the smallest eigenvalue of the adjacency matrix of the Cayley graph of $S$. Recall that the adjacency matrix of $\G^*_S$ is $A\big(S \cup (-S) \big)$, and hence
\[
\lambda\big( \G^*_S \big) \; = \; \lambda\big( S \cup (-S) \big).
\]
We shall use the following lemma, which bounds $\lambda(\G^*_S)$ in terms of $\lambda(S)$.

\begin{lemma}\label{SuS}
Let $0 \not\in S \subseteq G$ and $\delta > 0$. If $\lambda(S) \ge (\delta - 1)|S|$, then 
\[
\lambda\big( \G^*_S \big) \,\ge\, \left( \frac{\delta}{2} - 1 \right) |S \cup (-S)|.
\]
\end{lemma}

\begin{proof}[Proof of Lemma~\ref{SuS}]
By Lemma~\ref{evectors}, the characters of $G$ are a basis of eigenvectors of both $A(S)$ and $A((-S) \setminus S)$. Thus, by Lemma~\ref{lambda},
\begin{eqnarray*}
  \lambda\big( \G^*_S \big) & = & \lambda\big( S \cup (-S) \big) \; = \; \lambda(S) + \lambda\big( (-S) \setminus S \big)\\
  & \ge & ( \delta - 1 ) |S| - |(-S) \setminus S| \; \ge \; \left( \frac{\delta}{2} - 1\right)|S \cup (-S)|,
\end{eqnarray*}
as required. The last inequality follows from the fact that $|(-S) \setminus S| \le |(-S)| = |S|$.
\end{proof}

We can now complete the proof of Lemma~\ref{SF3_small}.

\begin{proof}[Proof of Lemma~\ref{SF3_small}]
Let $I \in \SF_\ge^{(\delta)}(G,m)$ and suppose that $\lambda(I) \ge (\delta - 1)|I|$. By Lemmas~\ref{chooseS} and~\ref{SuS}, there exists a set $S \subseteq I$ with $|S| = \eps m$, such that
\[
\lambda(\G^*_S) \,\ge\, \left( \frac{\delta}{4} - 1 \right) |S \cup (-S)|.
\]
Since $I$ is sum-free, $I \setminus S$ is an independent set in $\G^*_S$. We claim that $\G^*_S$ satisfies the conditions of Theorem~\ref{thm:graphs}. Indeed, $\G^*_S$ is a $d_S$-regular graph on $n$ vertices, where $d_S = |S \cup (-S)|$, and 
$$|I \setminus S| \,=\, \big( 1 - \eps \big) m \,\ge\,  \frac{C(\eps)n}{\eps m} \, \ge\, \frac{C(\eps)n}{|S \cup (-S)|},$$
since $m \ge C\sqrt{n}$. Note also that 
$$\ds\frac{ |\lambda(\G^*_S)| }{ d_S + |\lambda(\G^*_S)| } \, \le \, \frac{ 1 - (\delta/4) }{ 2 - (\delta/4) } \, \le \, \frac{1}{2} - \frac{\delta}{16}.$$ 
Hence, by Theorem~\ref{thm:graphs},
$$I(\G^*_S,(1-\eps)m) \,\le\, {(1/2-\delta/20)n \choose (1-\eps)m},$$
and so
$$\left| \Big\{ I \in \SF_\ge^{(\delta)}(G,m) \,\colon\, \lambda(I) \le (1 - \delta)|I| \Big\} \right| \, \le \, {n \choose \eps m}{(1/2-\delta/20)n \choose (1-\eps)m} \, \le \, 2^{-\eps m} {n/2 \choose m}$$
if $\eps = \eps(\delta) > 0$ is sufficiently small. This proves the lemma.
\end{proof}

Finally, note that Lemmas~\ref{SF3_big} and~\ref{SF3_small} imply Proposition~\ref{prop:SF2}.

\subsection{Proof of Theorem~\ref{q=2}}

First, observe that if $|G|$ is even, then the claimed lower bound on $|\SF(G,m)|$ is a straightforward consequence of the fact that, by Theorem~\ref{thm:SFG-structure}, \( |\SF_0(G)| = \#\big\{ \text{elements of $G$ of order $2$}\big\} \) and that each pair of distinct $B, B' \in \SF_0(G)$ intersects in $|G|/4$ elements. If $|G|$ is odd, then Theorem~\ref{q=2} only gives an upper bound on $|\SF(G,m)|$.

For the upper bound, observe that by Propositions~\ref{prop:SF1} and~\ref{prop:SF2}, we have 
\begin{eqnarray*}
  |\SF(G,m)| & \le & |\SF_\le^{(\delta)}(G,m)| \,+\, |\SF_\ge^{(\delta)}(G,m)|\\
  & \le & \left( \# \big\{ \text{elements of $G$ of order $2$} \big\} +o(1) \right) {{n/2} \choose m} \,+\, 2^{-\eps m} {{n/2} \choose m},
\end{eqnarray*}
for every $m \ge 4 \sqrt{n \log n}$, as required.

\end{document}